\documentclass[a4paper, 12pt, DIV=10]{scrartcl}

\usepackage[utf8]{inputenc}
\usepackage[T1]{fontenc}
\usepackage[english]{babel}
\usepackage{graphicx,xcolor}
\usepackage{enumerate}
\usepackage{amsmath, amsfonts, amssymb,amsthm}

\usepackage{libertine}
\usepackage[libertine]{newtxmath}
\usepackage[lining]{FiraSans}
\setkomafont{disposition}{\firamedium}

\usepackage[colorlinks=true, linkcolor=blue, filecolor=magenta, urlcolor=cyan, citecolor=gray]{hyperref}
\usepackage{cleveref}
\crefname{chapter}{Chapter}{Chapters}
\crefname{section}{Section}{Sections}
\crefname{subsection}{Section}{Sections}
\crefname{subsubsection}{Section}{Sections}
\crefname{figure}{Figure}{Figures}
\crefname{table}{Table}{Tables}
\crefname{equation}{}{}

\numberwithin{equation}{section}
\theoremstyle{definition}

\crefname{question}{Question}{Questions}

\crefname{str}{Strategy}{Strategies}
\newtheorem{definition}{Definition}[section]
\crefname{definition}{Definition}{Definitions}

\crefname{ex}{Example}{Examples}

\crefname{rmk}{Remark}{Remarks}
\theoremstyle{plain}
\newtheorem{thm}[definition]{Theorem}
\crefname{thm}{Theorem}{Theorems}

\crefname{conj}{Conjecture}{Conjectures}

\crefname{lemma}{Lemma}{Lemmas}

\crefname{cor}{Corollary}{Corollaries}

\crefname{claim}{Claim}{Claims}
\newtheorem{prop}[definition]{Proposition}
\crefname{prop}{Proposition}{Propositions}
\newtheorem{obs}[definition]{Observation}
\crefname{obs}{Observation}{Observations}

\newcommand{\N}{\mathbb{N}}

\newcommand{\cF}{\mathcal{F}}

\newcommand{\cI}{\mathcal{I}}

\newcommand{\cL}{\mathcal{L}}
\newcommand{\cP}{\mathcal{P}}

\author{Jan Corsten\thanks{Department of Mathematics, LSE, London WC2A 2AE, Email: \href{j.corsten@lse.ac.uk}{j.corsten@lse.ac.uk}}}
\date{\today}
\title{A note on the grid Ramsey problem\footnote{An extended abstract of this note has been published in Electronic Notes in Discrete Mathematics \textbf{61} (2017), 287--292.}}

\begin{document}

\maketitle

\begin{abstract}
The \emph{grid Ramsey number} $ G(r) $ is the smallest number $ n $ such that every edge-colouring of the grid graph $\Gamma_{n,n} := K_n \times K_n$ with $r$ colours induces a rectangle whose parallel edges receive the same colour. We show $ G(r) \leq r^{\binom{r+1}{2}} - \left( 1/4 - o(1) \right) r^{\binom{r}{2}+1} $, slightly improving the currently best known upper bound due to Gy\'arf\'as.
\end{abstract}

%%%%%%%%%%%%%%%%%%%%%%%%%%%%%%%%%%%%%%%%%%%%%%%%%%%%%%%%%%

\section{Introduction}
Hales-Jewett's theorem is unquestionably one of the most important results in Ramsey theory; in \cite{RamseyTheory} Graham, Rothschild and Spencer  write that ``without it, Ramsey theory would more properly be called Ramseyan theorems''. To state the theorem we need to introduce some notation. Given positive integers $a$ and $n$, a \emph{combinatorial line} is a set of the form \[ \{ x \in [a]^n : x_i = x_j \text{ for all } i,j \in I \text{ and } x_i = a_i \text{ for all } i \not \in I \}, \] where $ I \subset [n] $ is a non-empty set and $a_i \in [a]$ is a fixed constant for every $ i \in I \setminus [n] $. Hales-Jewett's theorem \cite{HalesJewett} states that for all positive integers $ a $ and $ r $, there exists an integer $ n $ such that every colouring of $ [a]^n $ with $r$ colours induces a monochromatic combinatorial line.

The Hales-Jewett number $ HJ(a,r) $ is the smallest number $ n $ for which the above assertion is true. Originally, Hales and Jewett used a double induction in a product-argument, which yielded in an upper bound for $HJ(a,r)$ of Ackermann type. In \cite{Shelah_HJ} Shelah presented a new proof giving a primitive recursive upper bound by avoiding this double-induction. The key lemma of the proof is known as ``Shelah's Cube Lemma''. Especially the easiest non-trivial case of this lemma, the so called \emph{grid Ramsey problem}, attracted many researchers in the following years. To state this problem we begin with some definitions.

\begin{definition}
	For positive integers $ m $ and $ n $, the \emph{grid graph} $ \Gamma_{m,n} $ is the graph product $ K_m \times K_n $, that is $ V \left(\Gamma_{m,n}\right) = [m] \times [n] $ and $ \{(i,j),(i',j')\} \in E \left(\Gamma_{m,n}\right)$ if and only if either $ i=j $ or $ i'=j' $. 
\end{definition}

In other words, $ \Gamma_{m,n} $ consists of $ n $ vertical copies of $ K_m $ and $ m $ horizontal copies of $ K_n $. A \emph{rectangle} is a set of four vertices of the form $ \left\{ (i,j), (i',j), (i,j'), (i',j') \right\} $, where $ 1 \leq i < i' \leq m$ and $ 1 \leq j < j' \leq n$. Given a colouring of the edges of the grid graph, we call a rectangle alternating if its parallel edges receive the same colour. The grid Ramsey problem is to determine how large $m$ and $n$ have to be to guarantee an alternating rectangle in every colouring of $ E(\Gamma_{m,n}) $ with $r$ colours.

\begin{definition}
For positive integers $m,n$, let $ g(m,n) $ be the smallest number $ r $ such that there exists an $r$-colouring of $E(\Gamma_{m,n})$ without alternating rectangles.
\end{definition}

The diagonal case $ m=n $ is of greatest interest to us. In this case it is more common to write $ G(r) $ for the smallest number $ n $ such that every $r$-colouring of $E(\Gamma_{n,n})$ induces an alternating rectangle.

Shelah originally proved $G(r) \leq r^{\binom{r+1}{2}}+1$, in fact he showed the stronger statement $g(r+1, r^{\binom{r+1}{2}}+1) \geq r+1 $.
To see this, let $m=r+1$, $ n=r^{\binom{r+1}{2}}+1$ and consider an arbitrary $r$-colouring $ \chi $ of $E(\Gamma_{m,n})$. Since the number of different $r$-colourings of $K_{r+1}$ is $n-1$, two columns $i<j$ are coloured identically. Furthermore, by the pigeonhole principle, two of the $m = r+1$ horizontal edges between columns $i$ and $ j$ have the same colour, resulting in an alternating rectangle.
A trivial bound in the other direction is given by $ g(m,n) \leq \min \{m,n\} $ and $ G(r) \geq r + 1 $.
To see this, assume without loss of generality that $m \leq n$ and colour every horizontal edge by its row-index.

Improving these bounds immediately became a widely studied problem but progress has been slow. In \cite{RamseyTheory} Graham, Rothschild and Spencer asked whether $G(r)$ grows polynomially in $r$ and wrote ``a polynomial upper bound on $G(r)$ might well lead to a towerian upper bound to $HJ(a,2)$. Even if not, it is a certainly interesting problem for its own sake''.

Although both upper and lower bound follow from easy arguments and are very far apart, improving these bounds appeared to be difficult. The first improvements were made by Heinrich \cite{Heinrich} and by Faudree, Gy\'arf\'as and Sz\H onyi \cite{Faudree} who showed that $ G(r) \geq \Omega(r^3) $. The first and so far only improvement to the upper bound was made by Gy\'arf\'as \cite{gyarfas} who showed $ G(r) \leq r^{\binom{r+1}{2}} - r^{\binom{r-1}{2}+1} +1 $ improving Shelah's result by a small additive term. 

For a long time there was no progress until Conlon, Fox, Lee and Sudakov \cite{CFLS_GridRamsey} showed recently that $ G(r) $ grows super-polynomially in $ r $, thus answering the question of Graham, Rothschild and Spencer in \cite{RamseyTheory}.

Here we prove the following theorem, which gives another small improvement to the upper bound.

\begin{thm} \label{thm:diag}
	Let $ r \geq 2 $,  $m =  r^{\binom{r+1}{2}} - \lfloor r/4 \rfloor \cdot r^{\binom{r}{2}} +1$ and $ n= \frac12 \cdot r^{\binom{r+1}{2}}$. Then $ g(m,n) \geq r+1 $ and in particular we have $ G(r) \leq r^{\binom{r+1}{2}} - \left( 1/4 - o(1) \right) r^{\binom{r}{2}+1} $.
\end{thm}

We will also prove the following result about a more off-diagonal case.

\begin{thm} \label{thm:offdiag}
	Let $ r \geq 2 $,  $m =  r^{\binom{r+1}{2}} - r^{\binom{r}{2}} +1$ and $ n= r^{r-1}(r^r - 1) + r + 1$. Then $ g(m,n) \geq r+1 $.
\end{thm}

%%%%%%%%%%%%%%%%%%%%%%%%%%%%%%%%%%%%%%%%%%%%%%%%%%%%%%%%%%%
%%%%%%%%%%%%%%%%%%%%%%%%%%%%%%%%%%%%%%%%%%%%%%%%%%%%%%%%%%%

\section{Basic Concepts}
We begin with an easy but helpful observation about the inverse relation of $G$ and $g$. 

\begin{obs}
	$ G $ and $ g $ are non-decreasing in all parameters and we have for all $ n,r \geq 1 $
	\begin{enumerate}[(i)]
		\item $ G\left(g(n,n)\right) \geq n + 1 $ and $ G\left(g(n,n)-1\right) \leq n $.
		\item $ g\left(G(r), G(r) \right) \geq r + 1 $ and $ g\left(G(r)-1, G(r)-1\right) \leq r $.
	\end{enumerate}
\end{obs}

It turns out that it is enough to consider only vertical edges. To understand this, we need the following definition.
Fix positive integers $ m,n,r $ and colours $ c_1, \ldots, c_r $ for the rest of this section.

\begin{definition}
  Let $ \chi $ be an $ r $-colouring of the vertical edges of $\Gamma_{m,n} $. For every $ i =1, \ldots, n $, define $ \chi_i $ to be the colouring of $ E(K_m) $ induced by the $ i $-th column of $\Gamma_{m,n}$.
  Furthermore, define a graph $ \mathcal G(\chi_i,\chi_j) $ on $ [m] $ with edges $ \{ e \in E(K_n): \chi_i(e) = \chi_j(e) \} $ for every $ 1 \leq i < j \leq n $.
\end{definition}

The following observation first appeared in \cite{CFLS_GridRamsey}, where it played a crucial role in their construction of a ``good'' edge-colouring of a super-polynomially sized grid graph.

\begin{obs}\label{obs:ext}
	An $r$-colouring $ \chi $ of all vertical edges of $ \Gamma_{m,n} $ is extendible to an $r$-colouring of all edges without alternating rectangles if and only if $ \mathcal G(\chi_i,\chi_j) $ is $ r $-colourable for every $ 1 \leq i < j \leq n $.
\end{obs}

\begin{proof}
	Let $ \chi $ be a colouring  of all vertical edges of $ \Gamma_{m,n} $. Fix $ 1 \leq i < j \leq n $ and note that it is possible to colour the $ m $ horizontal edges between columns $ i $ and $ j $ without creating an alternating rectangle if and only if the graph $ \mathcal G(\chi_i,\chi_j)$ is $ r $-colourable.
\end{proof}

We will only work with colourings of the vertical edges from now on and see them as vectors $ \chi = (\chi_1, \chi_2, \ldots, \chi_n) $ of $ r $-colourings of $ E(K_m) $. For convenience, let $ C_r(m,n) $ be the set of such colourings and call $ \chi \in C_r(m,n) $ \emph{good} if $ \mathcal G(\chi_i,\chi_j) $ is $ r $-colourable for every $ 1 \leq i < j \leq m $, that is if we can extend it to a colouring of $ E(\Gamma_{m,n}) $ without alternating rectangles.

In their proofs, both Shelah \cite{Shelah_HJ} and Gy\'arf\'as \cite{gyarfas} only used the much weaker condition than guaranteed by \cref{obs:ext} that $ \mathcal G(\chi_i,\chi_j) $ must be $ K_{r+1} $-free in a good colouring $\chi $. We will use $ \chi(\mathcal G(\chi_i,\chi_j)) \leq r $ in a more crucial way in order to make an improvement to the upper bound.

\begin{definition}
	A colouring $ \chi' \in C_r(m,n) $ is obtained from $ \chi \in C_r(m,n) $ by \emph{switching} two colours $ c $ and $ \tilde c $ at an edge $ e $ if, for every $ i= 1, \ldots, n $, we have $ \chi'_i(e) = \tilde c $ whenever $ \chi_i(e) =  c$, $ \chi'_i(e) = c $ whenever $ \chi_i(e) =  \tilde c$ and $ \chi = \chi' $ everywhere else. For $ \chi,\chi' \in C_r(m,n) $, we then say that $ \chi \sim \chi' $ if $ \chi' $ can be obtained from $ \chi $ by finitely many switches.
\end{definition}

It is easy to see that $ \chi $ is good if and only if $ \chi' $ is good whenever $ \chi \sim \chi'$.

\begin{definition}
	Let $ k \in [r] $. We call a colouring $ \chi \in C_r(m,n) $ \emph{$ k $-stabilised} if $ \chi_i \equiv c_i $ for all $ i = 1, \ldots, k $. Furthermore, define $ g_k(r,n) $ to be the minimum number $ r $, such that there is a good, $ k $-stabilised colouring $ \chi \in C_r(m,n) $.
\end{definition}
Observe that we can always stabilise a colouring for the first colour: Given a good colouring $ \chi \in C_r(m,n) $ we obtain a good $ 1 $-stabilised colouring $ \chi'\in C_r(m,n) $ by switching $ c_1 $ and $ \chi_1(e) $ at $ e $ for every $ e \in E(K_m) $. This implies the following observation.

\begin{obs}\label{obs:stab1}
	For every $ m,n \geq 1 $, we have $ g(m,n) = g_1(m,n)$.
\end{obs}

%%%%%%%%%%%%%%%%%%%%%%%%%%%%%%%%%%%%%%%%%%%%%%%%%%%%%%%%%%%%%%%%%%%%%%%%%%
%%%%%%%%%%%%%%%%%%%%%%%%%%%%%%%%%%%%%%%%%%%%%%%%%%%%%%%%%%%%%%%%%%%%%%%%%%

\section{The Proof}

For every positive integer $ r $, fix colours $ c_1, \ldots, c_r $ for the rest of this section. The main part of the proof is the following recursive bound for $g_k$.

\begin{prop}\label{prop:gk}
	For all $ m,n \in \N $ and every $ r < n $ we have
	\begin{enumerate}[(i)]
		\item $ g_{k+1}(m+1,n) \geq r+1 \implies g_{k}(mr^k+1,n) \geq r+1 $ for every $ k \leq r -1 $ and
		\item $ g_{r}(r^r+1,n) \geq r+1$.		
	\end{enumerate}
\end{prop}

Given a graph $G$, an r-colouring $\chi$ of $ E(G) $ and a set of colours $ C $, we call a set $I \subset V(G)$ \emph{$C$-independent} if $ \chi (e) \not \in C $ for all $e \in E(I)$.

\begin{proof}
	We prove the contrapositive of $(i)$. Assume $ g_k(mr^k +1,n) \leq r$, that is there exists a good colouring $ \chi \in C_r(mr^k+1,n) $ which is $k$-stabilised. For $ 1 \leq i \leq k $, let $ G_i $ be the subgraph of $ K_{mr^k+1} $ obtained by restricting to all edges $ e \in E(K_m) $ with $ \chi_{k+1}(e) = c_i $. Since $ \chi $ is $k$-stabilised, we have $ G_i = \mathcal G(\chi_i,\chi_{k+1}) $ and deduce from \cref{obs:ext} that $ G_i $ is $ r $-colourable for all $ i=1, \ldots, k $. Therefore, we find for every $ i = 1, \ldots, k $ a partition $ \cI_i$ of $[mr^k+1]$ into at most $r$ $c_i$-independent sets with respect to the colouring $\chi_{k+1} $. Let \[ \cP := \bigwedge_{i=1}^{k} \cI_i := \left\{ I_1 \cap \ldots \cap I_k : I_i \in \cI_i \text{ for all }  i \in [k] \right\} \] be their common refinement. $ \cP$ partitions $[mr^k+1] $ into at most $r^k$ $\{c_1, \ldots, c_k \}$-independent sets with respect to $ \chi _{k+1}$. Thus, by the pigeonhole principle, there is some $ X \in \cP $ with $ |X| \geq m +1 $. Now obtain $ \chi' $ from $ \chi $ by switching $ c_{k+1} $ and $ \chi(e) $ at $ e $ for every $ e \in \binom{X}{2} $ and note that this does not change any previously fixed colours. Restricting $ \chi' $ to $ X $ (and relabelling) gives a good $ (k+1) $-stabilised colouring and hence $ g_{k+1}(m+1,n) \leq r $, finishing the proof.
	
	The proof of $(ii)$ is very similar. In fact, the above proof works for $k = r$, noting that every $\{c_1, \ldots, c_r\}$-independent set must be of size at most $1$. 
\end{proof}

By iterating \cref{prop:gk} we easily deduce
\begin{align*}
g_{r}(r^r+1,r+1) \geq r+1 &\implies  g_{r-1}(r^r\cdot r^{r-1}+1,r+1) \geq r+1 \\ &\implies \ldots \implies g_{1}(r^{\binom{r+1}{2}}+1,r+1) \geq r+1,
\end{align*}
hence obtaining precisely Shelah's original bound. To prove \cref{thm:diag} and \cref{thm:offdiag} we make an improvement in the very first step of the iteration using a similar trick as Gy\'arf\'as in \cite{gyarfas}.

\begin{prop}\label{prop:offdiag}
	Let $ r \in \N$, $ m = r^{r-1}(r^r-1) $ and $n=r^{r-1}(r^r-1)+r+1$. Then $ g_{r-1}(m+1,n) \geq r+1 $.
\end{prop}

\begin{proof}
	Assume for contradiction that there is a good, $ (r-1) $-stabilised colouring $ \chi \in C_r(m+1,n) $ and fix a column index $ r \leq j \leq n $. Proceeding as in the proof of \cref{prop:gk}, we find a partition $ \cP_j $ of $ [m+1] $ into at most $ r^{r-1} $ $\{c_1, \ldots, c_{r-1} \}$-independent sets (in $\chi_j)$. By the pigeonhole-principle, there is some $ F_j \in \cP_j $ of size at least $ r^r $.
	%Since every $I \in \cP_j$ induces an $r$-stabilised good colouring, $ F_i $ must be of size exactly $ r^r $ by item $(ii)$ of \cref{prop:mk}.
	
	We show that $ \{F_r, \ldots, F_n \} $ is $ r $-intersecting, i.e.\ $ |F_i \cap F_j| = r $ for every $ r \leq i < j \leq n $. Fix some $ r \leq i < j \leq n $. We have $ | I_1 \cap I_2 | \leq r $ for every $ I_1 \in \cP_i $ and every $ I_2 \in \cP_j$, because $ I_1 \cap I_2 $ induces a $ c_r $-monochromatic clique in $ G(\chi_i,\chi_j) $. Moreover, we have $ F_i = \bigcup_{I \in \mathcal P_j} F_i \cap I $ and hence \[ r^r \leq |F_i| = \sum_{I \in \mathcal P_j} |F_i \cap I| \leq |\mathcal P_j| r \leq r^r,\] forcing equality everywhere. In particular $ |F_i \cap F_j| = r $. By Fisher's inequality such a family has size at most the number of elements in the ground set, that is $ n-r+1 \leq m + 1$ and hence $ n \leq m + r $, a contradiction.
\end{proof}

Starting the iteration with \cref{prop:offdiag}, we deduce \cref{thm:offdiag}. This already gives a new upper bound to $ G(r) $, but the following proposition leads to a stronger diagonal result.

\begin{prop}\label{prop:diag}
	Let $ r \in \N$, $ m = r^{r-1}(r^r-\lfloor r/4 \rfloor) $ and $n=\frac12 r^{\binom{r+1}{2}}$. Then $ g_{r-1}(m+1,n) \geq r+1 $.
\end{prop}

We will use the non-uniform Ray-Chaudhuri -- Wilson theorem (due to Frankl and Wilson \cite{FranklWilson}).

\begin{thm}[Non-uniform RW theorem]\label{thm:fw}
Let $n$ and $\ell$ be positive integers with $\ell \leq n$, $ \mathcal L \subset [n] $ be a set of size $\ell$ and $ \mathcal F \subset 2^{[n]} $ be a family of sets. If $ \mathcal F $ is $ \mathcal L $-intersecting (that is $ |F_1 \cap F_2 | \in \mathcal L $ for any two distinct $ F_1,F_2 \in \mathcal F$), then \[ | \mathcal F| \leq \sum_{i=0}^\ell \binom{n}{i} .\]
\end{thm}

\begin{proof}[Proof of \cref{prop:diag}.]
	Assume for contradiction that there is a good, $ (r-1) $-stabilised colouring $ \chi \in C_r(m+1,n) $ and fix a column index $ r \leq j \leq n $. Proceeding as in the proof of \cref{prop:gk}, we find a partition $ \cP_j $ of $ [m+1] $ into at most $ r^{r-1} $ $\{c_1, \ldots, c_{r-1} \}$-independent sets (in $\chi_j$). By the pigeonhole-principle, there is some $ F_j \in \cP_j $ of size at least $ r^r - \lfloor r/4 \rfloor + 1 $.
    
    We show that $ \cF = \{F_r, \ldots, F_n \} $ is $ \cL $-intersecting for $ \cL = \{r- \lfloor r/4 \rfloor +1 , \ldots, r \} $. We have $ | I_1 \cap I_2 | \leq r $ for every $ i \not = j $, $ I_1 \in \cP_i $ and $ I_2 \in \cP_j$, because $ I_1 \cap I_2 $ induces a $ c_r $-monochromatic clique in $ G(\chi_i,\chi_j) $. Hence, if $ \cF$ were not $\cL$-intersecting, there were indices $i$ and $j$, $ r \leq i < j \leq n $, with $| F_i \cap F_j| \leq r - \lfloor r/4 \rfloor $. Since $ F_i = \bigcup_{I \in \mathcal P_j} F_i \cap I $, we then have $ |F_i| \leq |\mathcal P_j| r - \lfloor r/4 \rfloor \leq r^r - \lfloor r/4 \rfloor$, a contradiction.
    
    By \cref{thm:fw}, $ \cF $ has size at most $ \sum_{i=0}^{\lfloor r/4 \rfloor} \binom{m}{i}$ and hence \[ n \leq  \sum_{i=0}^{\lfloor r/4 \rfloor} \binom{m}{i} + r - 1     < \frac12 r^{\binom{r+1}{2}},\] a contradiction.
\end{proof}

Finally, starting the iteration with \cref{prop:diag}, we deduce \cref{thm:diag}.

\section{Acknowledgement}
Significant part of the research was done when the author was participating in the Ramsey DocCourse programme in Prague 2016. He would like to thank Jaroslav Nešetřil and Jan Hubička for organising the course and David Conlon for introducing the problem during his lectures. The author would also like to thank Peter Allen, Julia B\"ottcher and Jozef Skokan for helpful comments on this note.

%%%%%%%%%%%%%%bibliography%%%%%%%%%%%%%%%%%%%%%%%%%%%%%%%%%%%%%%
\providecommand{\bysame}{\leavevmode\hbox to3em{\hrulefill}\thinspace}
\providecommand{\MR}{\relax\ifhmode\unskip\space\fi MR }
\providecommand{\MRhref}[2]{%
  \href{http://www.ams.org/mathscinet-getitem?mr=#1}{#2}
}
\providecommand{\href}[2]{#2}


\begin{thebibliography}{1}

\bibitem{CFLS_GridRamsey}
D.~Conlon, J.~Fox, C.~Lee, and B.~Sudakov, \emph{On the grid {R}amsey problem
  and related questions}, Int. Math. Res. Not. IMRN (2015), 8052--8084.

\bibitem{Faudree}
R.~J.~Faudree, A.~Gy\'arf\'as, and T.~Sz\H{o}nyi, \emph{Projective spaces and
  colorings of {$K_m\times K_n$}}, Sets, graphs and numbers ({B}udapest, 1991),
  Colloq. Math. Soc. J\'anos Bolyai, vol.~60, North-Holland, Amsterdam, 1992,
  pp.~273--278.

\bibitem{FranklWilson}
P.~Frankl and R.~M.~Wilson, \emph{Intersection theorems with geometric
  consequences}, Combinatorica \textbf{1} (1981), 357--368.

\bibitem{RamseyTheory}
R.~L.~Graham, B.~L.~Rothschild, and J.~H.~Spencer, \emph{Ramsey theory},
  Wiley-Interscience series in discrete mathematics and optimization, J. Wiley
  \& sons, New York, Chichester, Brisbane, 1990.

\bibitem{gyarfas}
A.~Gy{\'a}rf{\'a}s, \emph{On a {R}amsey type problem of {S}helah}, Extremal
  problems for finite sets. Conference. Visegr{\'a}d, 1991. (Bolyai Society
  mathematical studies 3.) (1994), 283--287.

\bibitem{HalesJewett}
A.~W.~Hales and R.~I.~Jewett, \emph{Regularity and positional games},
  Transactions of the American Mathematical Society \textbf{106} (1963),
  222--229.

\bibitem{Heinrich}
K.~Heinrich, \emph{Coloring the edges of {$K_m\times K_m$}}, J. Graph Theory
  \textbf{14} (1990), 575--583.

\bibitem{Shelah_HJ}
S.~Shelah, \emph{Primitive recursive bounds for van der {W}aerden numbers},
  Journal of the American Mathematical Society \textbf{1} (1988), 683--697.

\end{thebibliography}
\end{document}